\documentclass[11pt]{amsart}
\usepackage{geometry}                
\usepackage{graphicx}
\usepackage{amssymb}
\usepackage{epstopdf}
\usepackage{amsmath}
\usepackage{color}
\usepackage{hyperref}
\usepackage{pdfsync}
\usepackage{tikz}

\usepackage[all]{xy}
\xyoption{matrix}
\xyoption{arrow}

\newtheorem{thm}{Theorem}[subsection]
\newtheorem{lem}[thm]{Lemma}
\newtheorem{cor}[thm]{Corollary}
\newtheorem{prop}[thm]{Proposition}

\theoremstyle{definition}
\newtheorem{defn}[thm]{Definition}

\theoremstyle{remark}
\newtheorem{rem}[thm]{Remark}
 
\numberwithin{equation}{section}

 \tikzset{help lines/.style={step=#1cm,very thin, color=gray},
help lines/.default=.5} 
\tikzset{thick grid/.style={step=#1cm,thick, color=gray},
thick grid/.default=1} 

\def\el{\ell}

\def\LL{\Lambda}

\def\a{\alpha}
\def\b{\beta}
\def\g{\gamma}

\def\e{\epsilon}
\def\f{\varphi}

\def\t{\tau}
\def\th{\theta}

\def\ul{\underline}

\newcommand{\xrarrow}{\xrightarrow} 

\newcommand{\ot}{\leftarrow}

 \newcommand{\into}{\hookrightarrow}

\DeclareMathOperator{\Hom}{Hom}%
\DeclareMathOperator{\Ext}{Ext}%
\DeclareMathOperator{\End}{End}%
 \DeclareMathOperator{\mesh}{mesh}  

\DeclareMathOperator{\Aut}{Aut}

\newcommand{\field}[1]{\mathbb{#1}}
\newcommand{\ZZ}{\ensuremath{{\field{Z}}}}

\newcommand{\RR}{\ensuremath{{\field{R}}}}

\newcommand{\commentout}[1]{}

\newcommand{\cB}{\ensuremath{{\mathcal{B}}}}
\newcommand{\cC}{\ensuremath{{\mathcal{C}}}}
\newcommand{\cD}{\ensuremath{{\mathcal{D}}}}

\newcommand{\cM}{\ensuremath{{\mathcal{M}}}}

\newcommand{\cP}{\ensuremath{{\mathcal{P}}}}

\newcommand{\cT}{\ensuremath{{\mathcal{T}}}}
\newcommand{\cU}{\ensuremath{{\mathcal{U}}}}
\newcommand{\cV}{\ensuremath{{\mathcal{V}}}}

\newcommand{\cX}{\ensuremath{{\mathcal{X}}}}

\title{Continuous cluster categories II:\\Continuous Cluster-Tilted Categories}
\author{Kiyoshi Igusa}
\address{Department of Mathematics, Brandeis University, Waltham, MA 02454}
\email{igusa@brandeis.edu}
\thanks{The first author was supported by NSA Grant \#H98230-13-1-0247 when this paper was written.}

\author{Gordana Todorov}
\address{Department of Mathematics, Northeastern University, Boston, MA 02115}
\email{g.todorov@northeastern.edu}
\thanks{The second author was supported by NSF Grants \#DMS-1103813, \#DMS-0901185, when this paper was written.}


\begin{document}

\begin{abstract}
We show that the quotient of the continuous cluster category $\cC_\pi$ modulo the additive subcategory generated by any cluster is an abelian category and we show that it is isomorphic to the category of infinite length modules over the endomorphism ring of the cluster. These theorems extend the theorems of Caldero-Chapoton-Schiffler and Buan-Marsh-Reiten for cluster categories to the continuous cluster categories of type $A$. These results will be generalized in a series of forthcoming joint papers of the two authors with Job Rock.
\end{abstract}

\maketitle

\tableofcontents



\section*{Introduction}

This paper is based on a lecture given by the first author at the University of Sherbrooke with the title ``Continuous cluster categories II'' and was presented at an AMS sectional meeting in Iowa under the title ``Continuous spaced-out cluster category.'' Both lectures followed lectures by the second author entitled ``Continuous cluster categories'' which explained the basic constructions and properties of the continuous derived category and continuous cluster categories. 

Since the first paper on this subject is becoming too long, we decided to start a second paper which begins with a review, just as in the lecture. This second paper will concentrate on the interpretation of objects in the continuous cluster-tilted category $\cC_\pi/\cT$ as modules over the  Jacobian algebra of an infinite quiver with potential. This was inspired by \cite{BZ11} which builds on an earlier paper {ABCP} and we thank Thomas Br\"ustle for explaining his work to the authors so enthusiastically.

Section 1 is a review of basic definitions. The \emph{continuous cluster category} is a triangulated category $\cC=\cC_\pi$ which has as indecomposable objects the points of an open Moebius band $\cM$. The automorphism group $\Aut(\cC)$ is an extension of the group of orientation preserving homeomorphisms of the circle $S^1=\RR/2\pi\ZZ$. A \emph{cluster} is a discrete maximal compatible subset $\cT$ of $\cM$. Since $\Aut(\cC)$ acts transitively on the set of clusters in $\cC$ up to isomorphism, we only consider the standard cluster $\cT_0$. The \emph{rational cluster category} $\cX$ is defined to be the additive full subcategory generated by objects which can be obtained from the standard cluster by a finite sequence of mutations. The \emph{rational cluster-tilted category} is the quotient $\cX/\cT_0$ and the \emph{continuous cluster-tilted category} is $\cC/\cT$.

In section 2 we use ``ends'' and ``supports'' to prove that these cluster-tilted categories are abelian. The objects of the continuous cluster category correspond to pairs of distinct points on the circle $S^1$ which we call its \emph{ends}. An end is called \emph{rational} if it is the end of an object of $\cT$. An object of $\cC$ is rational (lies in $\cX$) if and only if its ends are rational. We do induction on the number of irrational ends of $Y$ to show that any morphism $X\to Y$ in $\cC/\cT$ has a kernel. Since $\cC/\cT$ is isomorphic to its opposite, it also has cokernels.

The \emph{support} of an object $X\in\cC/\cT$ is the set of objects $S$ in the cluster $\cT$ so that 
\[
	\Hom_{\cC/\cT}(\t^{-1}S,X)\neq0
\]
where $\t^{-1}S$ is the ``infinitesimal Auslander-Reiten translation'' of $S$. An object of $\cC/\cT$ lies in $\cX/\cT$ if and only if its support is finite. We show that a morphism $X\to Y$ in $\cC/\cT$ is a monomorphism or epimorphism if and only if it is a monomorphism or epimorphism on supports. To show that the cluster-tilted categories $\cX/\cT$ and $\cC/\cT$ are abelian, we show that a morphism is an isomorphism if and only if it is an isomorphism on supports.

In the third section of this paper we identify $\cX/\cT$ and $\cC/\cT$ with categories of string modules over the Jacobian algebra $\LL$ of an infinite quiver with potential. By Butler and Ringel \cite{BR87} every finite subquiver generates a string algebra. This extends without much trouble to the infinite case. The vertices in the infinite quiver correspond to the objects of the cluster $\cT$ and an indecomposable object in $\cC/\cT$ corresponds to the unique string module with the same support. This gives an isomorphism between the category $\cX/\cT$ and the category of finite length modules over $\LL$. This correspondence also gives an isomorphism between $\cC/\cT$ and the category $Rep_0\LL$ of infinite string modules without an infinite sequence of inward pointing arrows.

\section{Review}

We recall the construction of the continuous cluster categories $\cC_c$ for $0<c\le\pi$ and we note that for $c'<c$ the category $\cC_{c'}$ is a quotient of $\cC_c$. We will then concentrate on the largest cluster category $\cC_\pi$. See \cite{IT09}, \cite{IT10} for details. All of our categories will be Krull-Schmidt $K$-categories which are strictly monoidal with respect to direct sum. In other words, all objects are finite formal sums of indecomposable objects which have local endomorphism rings. We will say that the category is ``strictly additive'' to emphasize that this is with respect to direct sum and not tensor product. All functors will also be strictly additive so that they are uniquely determined by their restriction to the full subcategory of indecomposable objects. Also, with one exception, all triangulated functors $F$ will be \emph{strictly triangulated} in the sense that $FT=TF$ and $F$ takes distinguished triangles to distinguished triangles. The exception is the triangulated embedding of the standard cluster category of type $A_n$ into the corresponding continuous cluster category.

\subsection{Construction of the continuous cluster categories}

Let $\cP$ be the strictly additive $K$-category with one indecomposable object $P_x$ for every real number $x$.  Morphism are given by
\[
	\Hom_\cB(P_x,P_y)=\begin{cases} K & \text{if $x\le y$}\\
    0& \text{otherwise}
    \end{cases}
\]
Composition is given by multiplication. We will refer to any additive full subcategory of $\cP$ as a \emph{linear category}. Each linear category is abelian and all exact sequences split.

Let $\cB$ be the strictly additive exact category given as follows. The indecomposable objects of $\cB$ are pairs $(P_x,P_y)$ which we denote simply by $(x,y)\in\RR^2$. Homomorphisms are given by
\[
	\Hom_\cB((x,y),(x',y'))=\begin{cases} K & \text{if $x\le x'$ and $y\le y'$}\\
    0& \text{otherwise}
    \end{cases}
\]
Composition is given by multiplication of scalars. The morphism corresponding to $1\in K$ is called the \emph{basic morphism}. A sequence $0\to A\to B\to C\to 0$ in $\cB$ is \emph{exact} if it is split exact in each coordinate, i.e., if its image under both projection functors $\pi_1,\pi_2:\cB\to\cP$ are split exact in $\cP$. For example, we have the nonsplit exact sequence:
\begin{equation}\label{nonsplit sequence in B}
	0\to (x,y)\xrarrow{\binom11} (x,y')\oplus (x',y)\xrarrow{(1,-1)} (x',y')\to0
\end{equation}
for any $x< x'$ and $y< y'$. This implies that the number of indecomposable summands of the middle term $B$ will always be equal to the number of summands in $A\oplus C$. The categories $\cP$ and $\cB$ can be interpreted as categories of representations of the ``continuous quiver'' $\RR$.

For any $c>0$ let $\cB_c$ be the additive full subcategory generated by all $(x,y)$ where $|y-x|\ge c$. The \emph{continuous derived category} $\cD_c$ is defined to be the full subcategory of the quotient category $\cB/\cB_c$ generated by the nonzero indecomposable objects of $\cB/\cB_c$ which are $(x,y)$ with $|y-x|<c$. Thus $\cD_c\cong \cB/\cB_c$ has exactly one object in every isomorphism class of objects in $\cB/\cB_c$. Since $\cB_c$ is an approximation subcategory for $\cB$, we have the following. (See \cite{BM94}.)

\begin{thm}
$\cD_c$ is a triangulated category. If $X=(x,y)$ then $TX=(y+c,x+c)$.
\end{thm}

We use the sign convention that ``positive triangles'' (or ``up-right-up'' triangles) have positive signs on their morphisms. Thus for any $x-c<y<z<x+c$ the sequence
\[
	(x,y)\xrarrow 1 (x,z)\xrarrow 1 (y+c,z)\xrarrow 1 (y+c,x+c),
\]
with all morphisms being basic morphism, is a distinguished triangle. This comes from the exact sequence (\ref{nonsplit sequence in B}) when $z=y'$ and $x'=y+c$ since $(x',y)=0$ in that case. Also the ``negative triangle'' (or ``right-up-right'' triangle)
\begin{equation}\label{negative triangle}
	(x,y)\xrarrow 1 (w,y)\xrarrow 1 (w,x+c)\xrarrow{-1} (y+c,x+c),
\end{equation}
which comes from (\ref{nonsplit sequence in B}) when $x'=w,y'=x+c$  is distinguished. Up to isomorphism, these are all the distinguished triangles in $\cD_c$ with each term being indecomposable.

Next, we create the \emph{doubled category} $\cD_c^{(2)}$ given by adding an additional copy $(x,y)'$ of every indecomposable object $(x,y)$ in $\cD_c$ together with an isomorphism $\eta:(x,y)'\cong (x,y)$ with the property that $T$ anticommutes with $\eta$. In other words, $T(X')=(TX)'$ but
\[
	T\eta=-\eta_{TX'}:TX'\to TX
\]
This implies that a negative triangle with positive signs:
\[
	(x,y)'\xrarrow 1 (w,y)\xrarrow 1 (w,x+c)\xrarrow{1} (y+c,x+c)'
\]
is a distinguished triangle since it is isomorphic to the triangle (\ref{negative triangle}). We say that $(x,y)'$ has the \emph{opposite parity} as $(x,y)$ and we define $(x,y)''=(x,y),\eta'=\eta^{-1}:(x,y)\cong(x,y)'$.

For any $d\ge c$ the continuous cluster category $\cC_{c,d}$ is defined to be the orbit category $\cC_{c,d}:=\cD_c^{(2)}/F_d$ where $F_d$ is the triangulated functor on the doubled category $\cD_c^{(2)}$ defined by
\[
	F_d(x,y)=(y+d,x+d)'
\]
Since $F_d$ takes positive triangles to negative triangles, the change in parity is necessary to make $F_d$ strictly triangular. We denote the orbit of $(x,y)$ in $\cC_{c,d}$ by $M(x,y)$. So, $M(x,y)=M(y+d,c+d)'$.

\begin{thm}
The category $\cC_{c,d}$ is triangulated so that the orbit map $\cD_c^{(2)}\to \cC_{c,d}$ is strictly triangulated and all distinguished triangles $A\to B\to C\to TA$ with indecomposable $A,B,C$ are images of distinguished triangles in $\cD_c^{(2)}$.
\end{thm}

It is easy to see that, up to isomorphism, the triangulated category $\cC_{c,d}$ depends only on the ratio $c/d$ in the sense that $\cC_{c,d}\cong \cC_{ac,ad}$ for any positive real number $a$. Therefore, we can fix $d$ to be any convenient number. So, we choose $d=\pi=3.14159\cdots$ and we use the notation:
\[
	\cC_c:=\cC_{c,\pi}.
\]

\begin{thm}
If $\dfrac c\pi=\dfrac{n+1}{n+3}$ then there is a triangulated embedding of the cluster category of type $A_n$ into $\cC_c$.
\end{thm}

When $c<\pi$ we define two indecomposable objects $X,Y$ of $\cC_c$ to be \emph{compatible} if $\Ext^1(X,Y)=0=\Ext^1(Y,X)$. Recall that $\Ext^1(Y,X):=Hom(Y,TX)$ in any triangulated category. When $c=\pi$ we have another type of cluster category.

\begin{thm}
For any positive $n$, there is a triangulated embedding of the spaced-out cluster category of type $A_n$ into $\cC_\pi$.
\end{thm}

We will not use the ``spaced-out cluster category'' in this paper, so we forgo the definition. We note that in $\cC_\pi$ we have 
\[
T(M(x,y))=M(y+\pi,x+\pi)=M(x,y)'.
\]
Thus $TX=X'$ for all objects $X$. Also $X''=X$ for all objects.

\subsection{Clusters}

Recall that two indecomposable objects $X,Y$ in $\cC_\pi$ are \emph{compatible} if either $\Hom(X,Y)=0$ or $\Hom(Y,X)=0$. The objects compatible with any object $X=M(x,y)$ are $M(a,b)$ and $M(a,b)'$ where either 
\begin{enumerate}
\item $a\le x$ and $b\ge y$ ($M(a,b)$ is ``northwest'' of $M(x,y)$) or
\item $a\ge x$ and $b\le y$ ($M(a,b)$ is ``southeast'' of $M(x,y)$).
\end{enumerate}

Next, we need a metric on the set of indecomposable objects of $\cC_\pi$. We use the 1-\emph{norm} or ``taxi-cab metric'' which is the length of the shortest path consisting entirely of vertical and horizontal line segments. So, the distance between two objects $X,Y$ is the minimum of all real numbers of the form $|x-a|+|y-b|$ if $X=M(x,y)$ and $Y=M(a,b)$ or $Y=M(a,b)'$. An \emph{open ball} with \emph{radius} $\e$ around a point $M(x,y)$ is the set of all objects isomorphic to $M(a,b)$ where $|x-a|+|y-b|<\e$. This defines the usual topology on the set of isomorphism classes of indecomposable objects of $\cC_\pi$ which we denote by $\cM$. This is the \emph{open Moebius band}
\[
	\cM\equiv \{(x,y)\in\RR^2\,|\, |y-x|<\pi\}/(x,y)\sim (y+\pi,x+\pi).
\]

A \emph{cluster} $\cT$ is defined to be a discrete maximal pairwise compatible set of nonisomorphic indecomposable objects of $\cC_\pi$. By \emph{discrete} we mean that for every $M(x,y)$ in $\cT$ has an open neighborhood that contains no other object of $\cT$. Discreteness implies that $\cT$ is at most countably infinite.

\begin{prop}
Given any object $X$ in any cluster $\cT$ in $\cC_\pi$ there are, up to isomorphism, exactly two distinguished triangles $X\to A\to B\to TX$ and $X\to C\to D\to TX$ with $A,B,C,D$ in $\cT$. If we delete $X$ from the cluster $\cT$ then, up to isomorphism, there is a unique object $X^\ast$ in $\cC_\pi$ not isomorphic to $X$ so that $\cT\backslash \{X\}\cup \{X^\ast\}$ is a cluster and this new object $X^\ast$ is given by the octagon axiom for triangulated categories:
\[\xymatrixrowsep{10pt}\xymatrixcolsep{10pt}
\xymatrix{
B\ar[dr] && A\ar[ll] &&&& B\ar[dd] && A\ar[dl]\\
	& X\ar[ur]\ar[dl] &&&\longrightarrow&&& X^\ast\ar[ul]\ar[dr]\\
	C\ar[rr] && D\ar[ul]&&&& C\ar[ur]&&D\ar[uu]
	}
\]
\end{prop}

An object $X\in\cM$ will be called \emph{rational} with respect to a cluster $\cT$ if it can be obtained by a finite sequence of mutations from $\cT$. The set of rational points is countable. So, mutation does not act transitively on the set of clusters.

Let $Homeo_+(S^1)$ denote the group of orientation preserving homoeomorphisms of the circle $S^1=\RR/2\pi\ZZ$. Any $\f\in Homeo_+(S^1)$ lifts to a homeomorphism $\tilde\f$ of $\RR$ so that $\tilde\f(x+2\pi)=\tilde\f(x)+2\pi$ for all real $x$. Define an action of this group on the cluster category $\cC_\pi$ by
\[
	\f M(x,y)=M(\tilde \f(x),\tilde\f(y-\pi)+\pi)
\]
This is independent of the choice of $\tilde\f$ and defines a triangulated automorphism of $\cC_\pi$. Furthermore, for any triangulated automorphism $\psi$ of $\cC_\pi$, there is a $\f\in Homeo_+(S^1)$ so that $\psi(X)\cong\f(X)$ for all objects $X$.

\begin{thm}
For any two clusters $\cT_1,\cT_2$ in $\cC_\pi$ and any two objects $T_1\in\cT_1$, $T_2\in\cT_2$ there is a $\f\in Homeo_+(S^1)$ so that $\f(\cT_1)\cong\cT_2$ and $\f(T_1)=T_2$.
\end{thm}

Therefore, up to isomorphism, the continuous cluster category $\cC_\pi$ has only one cluster. When we need specific coordinates for objects we will take the \emph{standard cluster} $\cT_0$ which is defined to be the set of objects with coordinates
\[
	\left(\frac{m\pi}{2^n},\pi+\frac{(m-1)\pi}{2^n}\right)
\]
for integers $n\ge0$ and $0\le m<2^{n+1}$. We need to delete $(\pi,\pi)=(0,0)'$ to avoid repetition. The other points are all nonisomorphic. The set of points which are rational with respect to $\cT_0$ are those points with coordinates $(a\pi/2^n,b\pi/2^n)$ for all integer $a,b,n$ with $n\ge0$ and $|a-b|<2^n$.

\begin{defn}
Let $\cX$ be the additive subcategory of $\cC_\pi$ generated by all objects which are rational with respect to the standard cluster $\cT_0$. Then the quotient category $\cX/\cT_0$ will be called the \emph{rational cluster-tilted category}. For any cluster $\cT$ the quotient category $\cC_\pi/\cT$ will be called a \emph{continuous cluster-tilted category}.
\end{defn}


\section{Cluster-tilted categories are abelian}

In this section we will show that $\cX/\cT_0$ and $\cC_\pi/\cT_0$ are abelian categories. We will always assume that $\cT=\cT_0$ is the standard cluster and that $\cC=\cC_\pi$. So, we drop these subscripts.


\subsection{Ends}

\begin{defn}
We define the \emph{ends} of an indecomposable object $M(x,y)$ of $\cC$ to be the elements $x,y+\pi\in S^1$. The set of ends of any object of $\cC$ is defined to be the union of the set of ends of its components. An end is called \emph{rational} if it has the form $a\pi/2^n$ where $a,n$ are integers with $n\ge0$. An end is \emph{irrational} if it is not of this form.
\end{defn}

Note that $M(x,y)'=M(y+\pi,x+\pi)$ with ends $y+\pi,x+2\pi=x$ in $S^1$. Therefore, isomorphic objects have the same ends. We recall the following.

\begin{prop}
An object of $\cC$ lies in $\cX$ if and only if its ends are all rational.
\end{prop}

We need the following trivial observation about the additive full subcategory of objects in $\cC$ with one fixed end. For any $z\in S^1$, let $\cP_z$ be the full subcategory of $\cC$ generated by all $M(z,y)$ where $z-\pi<y<z+\pi$. Then $\cP_z$ is a linear category (isomorphic to an additive full subcategory of the category $\cP$). We assume that $z$ is irrational so that $\cP_z$ has no objects in the standard cluster $\cT$ and $\cP_z$ is isomorphic to its image in the quotient category $\cC/\cT$.

\begin{prop}\label{two cases for maps in P_z}
Let $f:X\to Y_0\oplus Y_1$ be any morphism between objects $X,Y_0,Y_1\in\cP_z$ where $Y_0$ is indecomposable. Let $f_0:X\to Y_0$ be the first component of $f$. Then there there is an automorphism $g$ of $X$ so that $f\circ g:X\to Y_0$ is nonzero on at most one component of $X$.
\end{prop}

\begin{proof}
If $f_0\neq0$, we take $X_0$ to be the terminal object in the collection of all summands of $X$ on which $f_0$ is nonzero. For every other summand $X_1$ of $X$, the restriction of $f_0$ to $X_1$ lifts to $X_0$ and we subtract this lifting from the identity map on $X$ to obtain the required automorphism $g$.
\end{proof}

\begin{defn} The \emph{mesh} of a finite subset of $S^1$ is defined to be the smallest positive difference between two elements. For any objects $X$ in $\cC$, the \emph{mesh} $\mesh(X)$ of $X$ is defined to be the mesh of the set of ends of $X$. We define $\mesh(0)=\pi$.
\end{defn}

\begin{lem} If $X,Y$ are indecomposable objects of $\cC/\cT$ then $\Hom_{\cC/\cT}(X,Y)\neq0$ if and only if 
\begin{enumerate}
\item $X\cong M(a,b), Y\cong M(x,y)$ for some $y-\pi<a\le x, x-\pi<b\le y$ and 
\item The closed rectangle $R=[a,x]\times [b,y]$ contains no points in the standard cluster $\cT$.
\end{enumerate}
Furthermore, the first condition is equivalent to the condition that $\Hom_\cC(X,Y)\neq0$.
\end{lem}

\begin{proof}
The first condition is equivalent to the condition that there is a morphism in $\cB$ from $X$ to $Y$ which does not factor through $\cB_\pi$. Equivalently, $\Hom(X,Y)\neq0$ in the continuous derived category and therefore in $\cC$. The closed rectangle $R$ is the set of all indecomposable objects $Z$ so that $\Hom_\cC(X,Z)$ and $\Hom_\cC(Z,Y)$ are both nonzero. Thus a nonzero morphism $X\to Y$ factors through $Z$ if and only if $Z\in R$. So, $X\to Y$ does not factor through any object of $\cT$ if and only if $\cT$ is disjoint from $R$. 
\end{proof}


\subsection{Infinitesimal $\t^{-1}$ and support}

For any object $S=M(x,y)\in\cT$ and any $0<\e<\frac12\mesh S$ the objects $S_\e:=M(x+\e,y+\e)$ and $S_{-\e}:=M(x-\e,y-\e)$ are nonzero objects of $\cC/\cT$. We consider the limit of $S_\e$ and of $S_{-\e}$ as $\e$ goes to zero. This will give us infinitesimal versions of the Auslander-Reiten translations $\t^{-1}S$ and $\t S$ of $S$. These play the respective roles of the projective and injective modules of the usual cluster-tilted algebras. 

\begin{defn}[infinitesimal Auslander-Reiten translation]
If $S\in\cT$ and $X$ is an object of $\cC$, we define
\[
	\Hom_{\cC/\cT}(\t^{-1}S,X):=\lim_{\e\to0+} \Hom_{\cC/\cT}(S_\e,X)
\]
and similarly,
\[
	\Hom_{\cC/\cT}(X,\t S):=\lim_{\e\to0+} \Hom_{\cC/\cT}(X,S_{-\e})
\]
The groups $\Hom_{\cC}(\t^{-1}S,X),\Hom_{\cC}(X,\t S)$ are defined analogously with $\cC/\cT$ replaced by $\cC$. When $\cC$ is replaced with $\cX$ we take $\e$ to be real numbers of the form $\pi/2^n$ with $n$ going to infinity.
\end{defn}

Strictly speaking, this only defines a covariant functor $\Hom_{\cC/\cT}(\t^{-1}S,-)$ on the category $\cC/\cT$ and similarly for $\t S$. We can also regard $\t^{-1}S,\t S$ as formal inverse and direct limits of sequences of objects in the category $\cC/\cT$.

\begin{lem}
There are natural duality isomorphisms
\[
	\Hom_\cC(\t^{-1}S,X)\cong D\Hom_\cC(X,S)
\]
\[
	\Hom_\cC(X,\t S)\cong D\Hom_\cC(S,X)
\]
for all $S\in\cT$ and $X\in\cC$ where $D=\Hom_K(-,K)$ is vector space duality. 
\end{lem}

\begin{proof}
The first duality is given by composition:
\[
	\Hom_\cC(\t^{-1}S,X)\otimes \Hom_\cC(X,S)\to \Hom_\cC(\t^{-1}S,S)=\lim_{\e\to0+}\Hom_\cC(S_\e,S)=K
\]
since every term in the limit is equal to $K$. The support of $\Hom_\cC(-,S)$ is the half open rectangle $(b-\pi,a]\times(a-\pi,b]$ if $S=M(a,b)$. But this is also the support of $\Hom_\cC(\t^{-1}S,-)$. So, we have a perfect duality. The other case is similar.
\end{proof}

\begin{prop}[infinitesimal Auslander-Reiten duality]
There is a natural duality:
\[
	\Hom_{\cC/\cT}(\t^{-1}S,X)\cong D\Hom_{\cC/\cT}(X,\t S)
\]
\end{prop}

\begin{proof}
This duality is given by the composition:
\[
	\Hom_{\cC/\cT}(\t^{-1}S,X)\otimes \Hom_{\cC/\cT}(X,\t S)\to\Hom_{\cC/\cT}(\t^{-1}S,\t S)= \lim_{\e,\delta\to0+}\Hom_{\cC/\cT}(S_\e,S_{-\delta})=K
\]
since every term in the limit is equal to $K$. If $S=M(a,b)$ then the support of $\Hom_{\cC/\cT}(\t^{-1}S,-)$ is equal to the open rectangle $(b-\pi,a)\times (a-\pi,b)$ which is also equal to the support of $\Hom_{\cC/\cT}(-,\t S)$. So, the duality is perfect. The boundary points which lie in the support of $\Hom_{\cC}(\t^{-1}S,-)$ are not in the support of $\Hom_{\cC/\cT}(\t^{-1}S,-)$ since $\cT$ contains points on the boundary arbitrarily close to the upper left corner $(b-\pi,b)$ and lower right corner $(a,a-\pi)$ and any morphism from $\t^{-1}S$ to any point in the boundary of the rectangle will factor through one of those points.
\end{proof}

We also need the infinitesimal analogue of the radical of the projective. Each infinitesimal radical has two components. 

\def\rt{r\t^{-1}}

\begin{defn} The \emph{infinitesimal radical} of $\t^{-1}S$ which we denote by $\rt S$ is formally defined as follows.
\[
	\Hom_{\cC/\cT}(\rt S,X):=\lim_{\e\to0+} \Hom_{\cC/\cT}(M(x+\e,y)\oplus M(x,y+\e),X)
\]
\end{defn}

\begin{prop}
Suppose that $S\in\cT$ and $X\in\cM$ and $S,X$ are not isomorphic. Then the following are equivalent.
\begin{enumerate}
\item $\Hom_\cC(S,X)\neq 0$ and $\Hom_\cC(X,S)\neq 0$.
\item There is a neighborhood $\cU$ of $S$ in $\cM$ so that $\Hom_\cC(T,X)\neq0$ for all $T\in \cU$.
\item There is a neighborhood $\cV$ of $S$ in $\cM$ so that $\Hom_\cC(X,T)\neq0$ for all $T\in \cV$.
\item $\Hom_{\cC/\cT}(\t^{-1}S,X)\neq 0$
\item $\Hom_{\cC/\cT}(X,\t S)\neq 0$.
\end{enumerate}
\end{prop}

\begin{proof}
Suppose that $S=M(a,b)$ is fixed. Then we will show that the set of all $X$ satisfying each of the above conditions is the same. By infinitesimal AR-duality, conditions (4), (5) are equivalent and they are both equivalent to the condition that $X$ lies in the open rectangle $(b-\pi,a)\times (a-\pi,b)$. But $\Hom_\cC(S,X)\neq 0$ if and only if $X$ lies in the half closed rectangle $[b-\pi,a)\times [a-\pi,b)$ and $\Hom_\cC(X,S)\neq 0$ if and only if $X$ lies in the half closed rectangle $(b-\pi,a]\times (a-\pi,b]$. The open rectangle is the intersection of these half closed rectangles. Therefore (4) and (5) are equivalent to condition (1).

Since $(b-\pi,a)\times (a-\pi,b)$ is an open set which varies continuously with $S$, any $X$ in this open rectangle will also satisfy the perturbations (2), (3) of the two conditions in (1). Conversely, each of conditions (2) and (3) implies condition (1). For example, (2) implies in particular that $\Hom_\cC(S,X)\neq0$ and it also implies that $\Hom_\cC(\t^{-1}S,X)\neq0$ which is equivalent to $\Hom_\cC(X,S)\neq0$ by the first duality lemma. Therefore, all conditions are equivalent.
\end{proof}

\begin{defn} The \emph{support} of an indecomposable object $X$ is defined to be the set of all objects $S\in\cT$ not isomorphic to $X$ satisfying any of the equivalent conditions given above. The \emph{support} of any object $X$ in $\cC$ is defined to be the union of the supports of its components. 
\end{defn}

Suppose that $X=M(x,y)$ lies in the rational cluster-tilted category $\cX$. Then we claim that the support of $X$ is finite. It follows from the definition that the support of $X$ is the set of all objects of $\cT$ which lie in the open rectangle $R=(y-\pi,x)\times (x-\pi,y)$. Suppose that $S=M(a\pi/2^n,b\pi/2^n)\in R\cap\cT$ where $a,b$ are integers with $b\ge a$. Then $b=a+2^n-1$ and $S\in R$ implies that $y-\pi<a\pi/2^n$ and $b\pi/2^n<y$. Adding $\pi-\pi/2^n$ to the first equation we get:
\[
	y-\frac{\pi}{2^n}<\frac{a\pi}{2^n}+\pi-\frac{\pi}{2^n}<y
\]
which implies that $y$ is not an integer multiple of $\pi/2^n$. Conversely, if $y$ is not an integer multiple of $\pi/2^n$ then we can find an integer $a$ satisfying the above inequality and this gives a point $S$ in $R\cap \cT$. If $y$ is not an integer multiple of $\pi/2^n$ for any $n$ then we get a sequence of objects in $R\cap \cT$ converging to the point $(y-\pi,y)$.

Going back to the original problem, suppose that there are an infinite number of points in $R\cap\cT$. Then they must converge to the boundary of the Moebius band. So, they converge either to the upper left corner $(y-\pi,y)$ or the lower right corner $(x,x-\pi)$. In the first case, $y$ cannot be equal to any integer times $\pi/2^n$ for any $n$. In the second case, $x$ cannot be equal to any integer times $\pi/2^n$ for any $n$. When $X=(x,y)$ lies in $\cX$, neither of these negative statement is true. Therefore, the support of $X$ is finite.

\begin{prop}
An object of $\cC$ lies in $add\,\cT$ if and only if its support is empty. The support of an object is finite if and only if the object lies in $\cX$.
\end{prop}

\begin{proof}
If $\Hom_\cC(S,T)$ and $\Hom_\cC(T,S)$ are both nonzero and $S,T$ are nonisomorphic, then $S,T$ are not compatible. Therefore, $S,T$ cannot both lie in $\cT$. So, the support of any object of $\cT$ is empty. Conversely, if either $\Hom_\cC(S,X)=0$ or $\Hom_\cC(X,S)=0$ for all $T\in\cT$ then $X$ is compatible with every object of $\cT$. So, $X$ lies in $add\,\cT$. This proves the first statement. The second statement was proved above.
\end{proof}

\begin{cor}
The support of any object $X$ of $\cC/\cT$ is well defined. The support of $X$ is empty if and only if $X=0$. The support of $X$ is finite if and only if $X\in\cX/\cT$.
\end{cor}

\subsection{Kernels in $\cX/\cT$} To show that the category $\cX/\cT$ has kernels, we need the right $add\,\cT$ approximation of each object in $\cX$.

\begin{prop}\label{T-approximation of objects of X}
If $X=M(x,y)$ is any object in $\cX$ which is not in $\cT$ then there is an exact sequence 
\[
0\to A\to B\to M(x,y)\to0
\]
in the exact category $\cB$ where $A,B$ are objects of the inverse image $\tilde\cT$ of $\cT$ in the open strip $\cU\subseteq\cB$. Furthermore, $B$ is a minimal right $add\,\cT$ approximation of $M(x,y)$ in $\cC$.
\end{prop}

The components of $A,B$ lie on a ``minimal walk'' in $\cT$ which we now define. A \emph{walk} from $Y$ to $Z$ in $\cT$ is a sequence of distinct objects $C_0=Y,C_1,\cdots,C_n=Z$ and irreducible basic morphisms between consecutive objects $C_i$ ($C_i\to C_{i+1}$ or $C_i\ot C_{i+1}$) where \emph{irreducible} means the morphism does not factor through any other object of $\cT$. The walk will be called \emph{nondegenerate} if the composition of any two composable morphisms in the walk is nonzero. The \emph{length} of a walk is the number $n\ge0$ of irreducible morphisms.

\begin{lem} Any minimal walk (walk of minimal length) is nondegenerate.
\end{lem}

\begin{proof}
If the composition of two irreducible morphisms $A\to B,B\to C$ is zero then we claim that there is an irreducible morphism $C\to A$. To prove this, we can apply an automorphism of $\cT$ and assume that $B=M(0,0)$, $A=(0,-\pi/2)$ and $C=(\pi/2,0)$. Then we have an irreducible morphism
\[
	C=\left(\frac\pi2,0\right)\to \left(\frac\pi2,\pi\right)'=\left(0,-\frac\pi2\right)=A
\] as claimed showing that the walk is not minimal
\end{proof}

Now we consider properties of a minimal walk from $Y$ to $Z\neq Y$. By an automorphism of $\cT$ any by possibly interchanging $Y,Z$ we may assume that $Y=M(0,0)$ and $Z=(a\pi/2^n,b\pi/2^n)$ where $-2^n<a\le0$, $b=2^n+a-1$ and $n\ge1$. Note that $a+b$ is odd. So, one of the coordinates of $Z$ is a reduced fraction (times $\pi$) and one is not. Thus $n=n(Z)$ is well defined.

There are exactly 4 irreducible maps at $Y$:
\[
	(0,0)\to \left(\tfrac\pi2,0\right), \quad(0,0)\to \left(0,\tfrac\pi2\right),\quad	\left(-\tfrac\pi2,0\right)\to(0,0),\quad \left(0,-\tfrac\pi2\right)\to(0,0)
\]
They all connect $Y$ to an object with $n=1$. Since $\Aut\cT$ acts transitively on objects, there are also exactly 4 irreducible maps at $Z$. Two increase $n$, one decreases $n$ and one keeps $n$ the same.
\[
\xymatrix{
& \left(\dfrac{2a\pi}{2^{n+1}},\dfrac{(2b+1)\pi}{2^{n+1}}\right)\ar[dl]\\
\left(\dfrac{(2a-1)\pi}{2^{n+1}},\dfrac{2b\pi}{2^{n+1}}\right)\ar[r]^(.6)f
&\left(\dfrac{a\pi}{2^{n}},\dfrac{b\pi}{2^{n}}\right)\ar[u]_g
	}
\]
This diagram shows all 4 cases if we let $Z$ be each of the three objects. (There are actually 6 cases depending on which of the coordinates of $Z$ is reduced.) We conclude that all walks from $Y=(0,0)$ to $Z$ have length at least $n$. Since there is an irreducible morphism connecting $Z$ to an object $W$ with $n(W)=n-1$, there is a walk with length equal to $n$. So, this walk must be minimal. Since $W$ is unique, we also conclude by induction on $n$ that the minimal walk from $Y$ to $Z$ is unique. We also note that the irreducible map connecting $Z,W$ is either horizontal $f:Z\to W$ or vertical:
\[\xymatrixrowsep{10pt}\xymatrixcolsep{10pt}
\xymatrix{
Z\\
W\ar[u]_g
	}
\]
This implies that all steps in the minimal walk from $Y$ to $Z$ go up and to the left (with arrows pointing up and right) in the sense that $C_{i+1}$ is located above or to the left of $C_i$ in the minimal walk $C_0=X,C_1,\cdots,C_n=Y$.

We record this and prove one more property of minimal walks.

\begin{lem}\label{lem:properties of minimal walk}
Any two objects $Y,Z$ of $\cT$ are connected by a unique minimal walk. The coordinates $Y=(c,d),Z=(a,b)$ can be chosen so that the walk goes from the lower right corner to the upper left corner of the closed rectangle
$
	R=[a,c]\times [d,b]
$ in vertical and horizontal steps as described above. Furthermore, this minimal walk goes through every object in $R\cap \cT$.
\end{lem}

\begin{proof}
It remains only to prove the last statement. By applying an automorphism of $\cT$ we may assume that $c=d=0$. Then $R=[a,0]\times [0,b]$. Suppose that $W=(x,y)$ is any object in $R\cap \cT$. Consider the unique minimal path from $Y$ to $W$. Since each step in this path goes up or to the right, this path lies in $R$. The first step in the minimal path from $Y=(0,0)$ to $W$ is equal to either $(-\pi/2,0)$ or $(0,\pi/2)$. But only one of these two points lies in the rectangle $R$. Therefore, the first step in the minimal path from $Y$ to $W$ is equal to the first step in the minimal path from $Y$ to $Z$. Similarly, each subsequent step in the two paths are equal. So, one path is contained in the other. In particular, $W$ lies in the unique minimal path from $Y$ to $Z$ as claimed.
\end{proof}

\begin{proof}[Proof of Proposition \ref{T-approximation of objects of X}]
Given $X=M(x,y)\in\cX$ let $a<x,b<y$ be maximal so that $Y=(x,b)$ and $Z=(a,y)$ lie in $\cT$. Consider the unique minimal walk from $Y$ to $Z$. By the lemma above, this walk is contained in the closed rectangle $R_X=[a,x]\times [b,y]$. Furthermore, all points of $R_X\cap\cT$ lie in this walk. Let $A_1,\cdots,A_n$ be the source points in this walk and let $B_0=Y,B_1,\cdots,B_{n-1},B_n=Z$ be the sink points in this walk. (See Figure \ref{T-approx} below.)
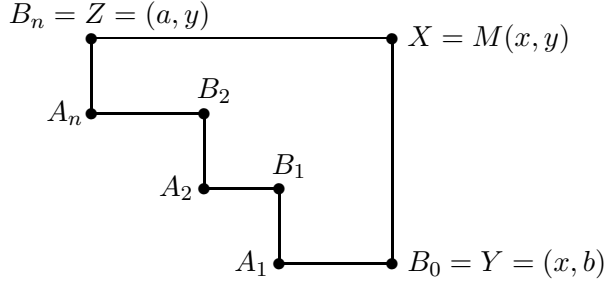
\begin{figure}[htbp]
\begin{center}
%
{
\setlength{\unitlength}{1cm}
{\mbox{
\begin{picture}(6,3.5)
\put(0,0.2){
      \thicklines
    \put(5,0){
         \put(-.1,-.1){$\bullet$}
           \put(.2,-.1){$B_0=Y=(x,b)$}
          \qbezier(0,0)(-.5,0)(-1.5,0)
          \thinlines
          \qbezier(0,0)(0,1)(0,3)    
          \thicklines
          }
    \put(3.5,0){
         \put(-.1,-.1){$\bullet$}
	\put(-.6,-.1){$A_1$}    
	          \qbezier(0,0)(0,.5)(0,1)    
	}
    \put(2.5,1){
         \put(-.1,-.1){$\bullet$}
	\put(-.6,-.1){$A_2$}    
	          \qbezier(0,0)(0,.5)(0,1)    
	          \qbezier(0,0)(.5,0)(1,0)
         \put(0.9,-.1){$\bullet$}
	\put(.9,0.2){$B_1$}	          
	}
    \put(1,2){
         \put(-.1,-.1){$\bullet$}
	\put(-.6,-.1){$A_n$}    
	          \qbezier(0,0)(0,.5)(0,1)    
	          \qbezier(0,0)(.5,0)(1.5,0)
         \put(1.4,-.1){$\bullet$}
	\put(1.4,0.2){$B_2$}	          
         \put(-.1,.9){$\bullet$}
	\put(-1.1,1.2){$B_n=Z=(a,y)$}  
	\thinlines          
         \qbezier(0,1)(1,1)(4,1)
         \put(3.9,.9){$\bullet$}
         \put(4.2,.9){$X=M(x,y)$}
	}
          }
\end{picture}}
}}
\caption{The walk associated to $X=M(x,y)$ with endpoints $B_0=Y$ and $B_n=Z$.}
\label{T-approx}
\end{center}
\end{figure}

The sum of basic morphisms $B_0\oplus B_1\oplus\cdots\oplus B_n\to X=M(x,y)$ is a right $add\,\cT$ approximation. To see this take any object $S\in\cT$ in the support of $\Hom_{\cC/\cT}(-,M(x,y))$. Then $S$ is southwest of $(x,y)$. Draw a straight line from $S$ to $(x,y)$. This line must pass through the walk from $Y=B_0$ to $Z=B_n$ since $S$ cannot be in the enclosed region by the lemma above. Follow the path either right or up to reach one of the $B_i$. This shows that any morphism $S\to M(x,y)$ factors through some $B_i$. Let $\bigoplus A_i\to \bigoplus B_i$ be the morphism given by taking the sum of the basic morphisms $A_i\to B_i$ minus the sum of the basic morphisms $A_i\to B_{i-1}$. Then
\[
	0\to\bigoplus A_i\to \bigoplus B_i\to M(x,y)\to0
\] is exact since it is split exact in each coordinate.
\end{proof}

\begin{prop}\label{X/T has kernels}
The rational cluster-tilted category $\cX/\cT$ has kernels and cokernels.
\end{prop}

\begin{proof}
Since $\cX/\cT$ is isomorphic to its opposite category, we only need to show that morphisms have kernels. We prove something stronger, namely, any morphism $f:X\to Y$ in $\cC/\cT$ with $Y\in\cX$ has a kernel in $\cC/\cT$ which lies in $\cX/\cT$ if $X\in\cX$.

Let $p:T^Y\to Y$ be the minimal contravariant $add\,\cT$ approximation of $Y$ as given by Proposition \ref{T-approximation of objects of X} above. Then the fiber is also in $add\,\cT$ and we have a distinguished triangle $T_1\to T^Y\to Y\to T_1'$ in the category $\cX$. If $A$ is the fiber of the composition $X\to Y\to T_1'$ then we have another distinguished triangle
\[
	T_1\to A\xrarrow j X\to T_1'
\]
which lies entirely in $\cX$ if $X\in\cX$. It follows from the octagon axiom that $A$ is the pull-back in the diagram:
\[
\xymatrix{
T_1\ar[d]\ar[r] &
	A\ar[d]\ar[r] &
	X\ar[d]^f\ar[r]
	& T_1'\ar[d]\\
T_1 \ar[r]& 
	T^Y \ar[r]^p&
	Y\ar[r]&T_1'
	}
\]
Equivalently, we have another distinguished triangle in $\cC$:
\[
	Y'\to A\to X\oplus T^Y\xrarrow {(f,p)} Y
\]

We claim that $A$ is the kernel of $f$ in $\cC/\cT$. The composition $A\to X\to Y$ is zero since it factors through $T^Y$. Suppose that $g:Z\to X$ is an morphism in $\cC$ so that $f\circ g:Z\to Y$ factors through some object $T$ of $\cT$. Then it factor through a map $h:Z\to T^Y$. Then $fg-ph=0$ so we can lift to $A$:
\[
\xymatrix{
	0\ar[r]\ar[d] &Z\ar[r]^=\ar@{-->}[d]_t & Z\ar[d]_{\binom g{-h}}\ar[r] & 0\ar[d]\\
	Y^-\ar[r]&A\ar[r] & X\oplus T^Y\ar[r]^ {(f,p)} &Y
}
\]
We still need to show that the morphism $t:Z\to A$ constucted above is unique as a morphism in $\cC/\cT$. This is equivalent to showing that $j:A\to X$ is a monomorphism modulo $\cT$. To do this we will show that if $j\circ t:Z\to X$ factors through an object of $add\,\cT$ then $t$ factors through an object of $add\,\cT$.

First we note that for $Z=S\in\cT$, the argument above shows that the mapping $S\to X$ factors through $A$. Take any morphism $t:Z\to A$ so that the composition $Z\xrarrow t A\xrarrow j X$ factors through an object of $S\in add\,\cT$. So, $jt=rs: Z\xrarrow s S\xrarrow r X$. Since $S\in add\,\cT$, $r$ lifts to a map $\tilde r:S\to A$ and $jt=j\tilde rs:Z\to X$. So, $j(t-\tilde rs)=0$. But this implies that $t-\tilde rs:Z\to A$ factors through the fiber $T_1$ of $j$. Since $T_1\in add\,\cT$ this implies that $t=\tilde rs$ in $\cC/\cT$. Since $\tilde rs$ factors through $S\in add\,\cT$ this implies that $t=0$ as a morphism in $\cC/\cT$.
\end{proof}

\subsection{Infinite walks for objects of $\cC/\cT$}

Suppose now that $X=M(x,y)$ is an object of $\cC$ which does not lie in $\cX$. Then one or both of the coordinates $x,y$ are not of the form $a\pi/2^n$. Thus, there are three cases:
\begin{enumerate}
\item $x=a\pi/2^n$ but $y$ does not have this form.
\item $y=b\pi/2^n$ and $x$ is not of this form.
\item Neither $x$ nor $y$ have the form $a\pi/2^n$.
\end{enumerate}
Case 1. Let $b<y$ be maximal so that $Y=M(x,b)\in\cT$. Then $Y\in R_X:=(y-\pi,x]\times [b,y)$.

\begin{lem}
In Case 1, there is a unique infinite walk starting at $Y$ which is contained in $R_X$. This infinite walk is locally minimal in the sense that any finite segment is minimal. Furthermore this infinite walk contains all objects in $R_X\cap \cT$ and each step of the walk is vertically up or horizontal to the left with arrows pointing up and right.
\end{lem}

\begin{proof}
By Lemma \ref{lem:properties of minimal walk}, every object in $R_X\cap \cT$ is connected by a unique minimal walk to $Y$ and this minimal walk lies in $R_X$. Furthermore, given any two such walks, one will be contained in the other. Therefore the union of all of these walks is an infinite locally minimal walk containing all objects of $R_X\cap\cT$. The other statement are local and follow from Lemma \ref{lem:properties of minimal walk}.
\end{proof}

Case 2. Let $a<x$ be maximal so that $Y=M(a,y)\in\cT$. Let $R_X:=[a,x)\times (x-\pi,y]$. Then the objects of $R_X\cap \cT$ are contained in an infinite walk which ends at the point $Y$.

Case 3. We take $R_X$ to be the open rectangle $R_X=:(y-\pi,x)\times (x-\pi,y)$. In this case there is a unique infinite walk containing objects arbitrarily close to the upper left corner and lower right corner of $R_X$. And this doubly infinite walk contains all objects in $R_X\cap\cT$.

In Case 3, the infinite walk associated to $X$ has no \emph{endpoints}. In Cases 1 and 2 the infinite walk has one endpoint called $Y$ in both cases. The walk associated to any $X\in\cX$ is finite with two endpoints except when $X\in\cT$ when the walk is empty. 

In all three cases, we let $A_i$ be the source points in the infinite walk and we let $B_i$ be the sink points. They should be indexed as before so that the irreducible map $A_i\to B_i$ is vertical pointing up and $A_i\to B_{i-1}$ is horizontal pointing to the right. Any endpoint is one of the $B_i's$.

\begin{prop}
The direct sum of the infinite sequences of objects $A_i$ and $B_i$ form a short exact sequence
\[
	0\to \bigoplus A_i\xrarrow j \bigoplus B_i\xrarrow p M(x,y)\to 0
\]
in the exact category $\cB$ and $T_0^X=\bigoplus B_i$ is a right $add\,\cT$ approximation for $X=M(x,y)$.
\end{prop}

\begin{proof}
We need to interpret this rigorously before we can prove it or use it correctly since the objects of $\cB$ are by definition only finitely generated.

$p:\bigoplus B_i\to X=M(x,y)$ being onto means that any $f:X\to Y$ which is zero on each $B_i$ must be zero. This is true since, in $\cB$, the composition of basic morphisms is basic and therefore nonzero.

When we say that the infinite sum $T_0^X=\bigoplus B_i$ is a right $add\,\cT$ approximation of $X$ we mean that any morphism of any object of $\cT$ into $X$ factors through one of the finite sums $\bigoplus_{|i|\le n}B_i$. This holds because any object $S\in \tilde\cT$ for which $\Hom_\cB(S,X)\neq0$ lies to the southwest of $X$ and the straight line from $S$ to $X$ crosses the infinite zig-zag in $R_X$ representing the infinite walk. Sliding forward along the arrow, or one of the arrows, at this intersection point we get to a point $B_i$ in $T_0^X$ and any morphism $S\to X$ factors through this $B_i$.

When we say that the infinite sum $\bigoplus A_i$ is the kernel of the map $p:\bigoplus B_i\to X$ we mean that, for any morphism $f: Z\to \bigoplus_{n\le i\le m}B_i$ so that the composition $p\circ f:Z\to X$ is zero, there is a unique morphism $g:Z\to \bigoplus_{n<i\le m}A_i$ so that $f=j\circ g$. This is an easy calculation. The morphism $f$ is the sum of its components $f_i:Z\to B_i$ which are scalar multiples, say $a_i$, times the basic morphism. The condition that $p\circ f=0$ is equivalent to saying that $\sum b_i=0$. If $n$ is minimal so that $b_n\neq0$ and $m$ is maximal so that $b_m=0$ then the object $Z$ must be southwest of both $B_n$ and $B_m$ which implies that it is southwest of $A_{n+1},\cdots,A_m$. The components $g_i:Z\to A_i$ of the lifting $g:Z\to \bigoplus_{n<i\le m}A_i$ are scalar multiples, say $a_i$, of the basic morphism $Z\to A_i$ and these scalars must satisfy the condition that $a_i-a_{i+1}=b_i$ and almost all $a_i$ are zero. It is easy to see that there is a unique solution of this system of equations.
\end{proof}

Since any morphism $S\to X$ in $\cC$ comes from a morphism in $\cB$ we get the following immediate corollary.

\begin{cor}
$T_0^X\to X$ is a right $add\,\cT$ approximation of $X$ in $\cC$.
\end{cor}

For any object $X\in\cC$ which is not in our cluster $\cT$, the rectangle $R_X$ is divided into two parts by the walk associated to $X$. These regions are the \emph{upper (right) part} $U_X$ which contains the object $X$ but does not contain any points in the continuous path representing the walk associated to $X$ and the \emph{lower (left) part} which by definition contains this walk. We have the following trivial observation.

\begin{lem}
An indecomposable object $Y\in C$ lies in $U_X$ if and only if $\Hom_{\cC/\cT}(Y,X)\neq0$.
\end{lem}

\begin{prop} A morphism $f:X\to Y$ in $\cC/\cT$ is zero in $\cC/\cT$ if and only if the induced map
\[
	\Hom_{\cC/\cT}(\t^{-1}S,f):\Hom_{\cC/\cT}(\t^{-1}S,X)\to \Hom_{\cC/\cT}(\t^{-1}S,Y)
\]
is equal to zero for all $S\in \cT$.
\end{prop}

\begin{proof}
Since a morphism is zero if and only if it is zero on each component, we may assume that $X,Y$ are indecomposable. Clearly $f=0$ implies $\Hom_{\cC/\cT}(\t^{-1}S,f)=0$. So, suppose that $f\neq0$ in $\cC/\cT$. By the lemma above, this implies that $X$ lies in $U_Y$. So the interior of the rectangle $R_X$ of $X$ meets the walk for $Y$ in at least one source point $S=A_i$. Then $S$ lies in the support of both $X$ and $Y$ and $\Hom_{\cC/\cT}(\t^{-1}S,f)\neq0$.
\end{proof}

\begin{cor}\label{cor: f is mono if it is mono on supports}
A morphism $f:X\to Y$ in $\cC/\cT$ is a monomorphism in $\cC/\cT$ if and only if the induced map $\Hom_{\cC/\cT}(\t^{-1}S,f)$ is a monomorphism for all $S\in \cT$.
\end{cor}

\begin{proof} If $f:X\to Y$ is a monomorphism in $\cC/\cT$ then the induced map
$
	\Hom_{\cC/\cT}(W,X)\to \Hom_{\cC/\cT}(W,Y)
$
is a monomorphism for all indecomposable $W$ not in $\cT$. This is also true for $W=\t^{-1}S$ since $S_\e$ is not in $\cT$ for all sufficient small $\e>0$. So, the first statement implies the second.

Conversely, suppose that $f$ is not a monomorphism. Then there exists a nonzero morphism $g:Z\to X$ so that $f\circ g:Z\to Y$ is trivial in $\cC/\cT$. By the above proposition, this implies that $f\circ g$ is trivial on supports but $g$ is not. Therefore $f$ is not a monomorphism on supports. I.e., $\Hom_{\cC/\cT}(\t^{-1}S,f)$ is not a monomorphism.
\end{proof}

\subsection{Kernels in $\cC/\cT$}

\begin{prop}\label{morphisms in C/T have kernels}
Any morphism $f:X\to Y$ in $\cC/\cT$ has a kernel and cokernel. 
\end{prop}

\begin{proof}
Since $\cC/\cT$ is isomorphic to its opposite category we only need to show that $f:X\to Y$ has a kernel. The proof will be by induction on the number of irrational ends of $Y$. If $Y$ has no irrational ends then $Y\in\cX/\cT$ and $f:X\to Y$ has a kernel by Proposition \ref{X/T has kernels}. So, suppose that $Y$ has at least one irrational end $z$. 

Let $Y=Y(z)\oplus Y_2$ where $Y(z)\cong\bigoplus M(z,y_j)$ is the sum of all components of $Y$ with one end at $z$. Similarly, we have $X=X(z)\oplus X_2$ where $X(z)\cong\bigoplus M(z,x_i)$. Then $Y(z)$ is nonzero. So it has at least one summand: $Y(z)=Y_0(z)\oplus Y_1(z)$ where $Y_0(z)=M(z,y_0)$. Let $f_0:X(z)\to Y_0(z)$ be the restriction to $X(z)$ of the projection of $f$ to $Y_0$. Then $f_0$ is a morphism in the category $\cP_z$. So, by Proposition \ref{two cases for maps in P_z} there are two possibilities: 
\begin{enumerate}
\item $f_0=0$ or 
\item There is a decomposition $X(z)=X_0(z)\oplus X_1(z)$ where $X_0=M(z,x_0)$ and an automorphism $g$ of $X(z)$ so that $f_0\circ g$ is zero on $X_1(z)$ and nonzero on $X_0(z)$.
\end{enumerate}

Let $\mu=\mesh(X\oplus Y)$. Then $z$ is the only end of $X\oplus Y$ in the open interval $(z-\mu,z+\mu)$. Also, all of the second coordinates $x_i,y_j$ of $X(z),Y(z)$ lie in the closed interval $C=[z-\pi+\mu,z+\pi-\mu]$. Now consider the set $z\times C$. This is a compact set which is disjoint from the cluster $\cT$ which is a closed subset of the open Moebius band. Therefore, there exists an $\e>0$ so that the closed set $[z-\e,z+\e]\times C$ has no points in the cluster $\cT$. By making $\e$ smaller if necessary, we can assume that $\e<\mu$ and that $z-\e=a\pi/2^n$ for integers $a,n$. Let $Y_0(z-\e)=M(z-\e,y_0)$. Then $Y_0(z-\e)$ has one fewer irrational end than $Y_0(z)$.

\ul{Claim 1}. We can choose $\e$ so that the basic morphism
\[
	h:Y_0(z-\e)=M(z-\e,y_0)\to Y_0(z)=M(z,y_0)
\]
is a monomorphism.

To prove this we recall that the walk associated to $Y_0(z)$ is infinite and has objects of $\cT$ converging to the point $(z,z-\pi)$. The first coordinates of these points converse to $z$ from the left. We choose $\e$ so that $z-\e$ is the first coordinate of some object $B_k$ (a component of the $add\,\cT$ approximation of $Y_0(z)$). Then the walk associated to $Y_0(z-\e)$ will be the portion of the walk for $Y_0(z)$ starting at $B_k$. The support of $Y_0(z-\e)$ will be the walk minus its endpoints. This will be a subset of the support of $Y_0(z)$ making the morphism $Y_0(z-\e)\to Y_0(z)$ into a monomorphism by Corollary \ref{cor: f is mono if it is mono on supports}.

\ul{Claim 2}. In the case $f_0=0$ we claim that the morphism $f:X\to Y=Y_0(z)\oplus Y_1(z)\oplus Y_2$ lifts to $\tilde Y=Y_0(z-\e)\oplus Y_1(z)\oplus Y_2$, i.e., $f=(h\oplus 1\oplus1)\circ \tilde f$.
\[\xymatrixrowsep{18pt}\xymatrixcolsep{8pt}
\xymatrix{
 &&& \tilde Y\ar@{=}[r]& Y_0(z-\e)\oplus Y_1(z)\oplus Y_2\ar@{>->}[d]^{h\oplus 1\oplus1}\\
X \ar[rrr]_f\ar[rrru]^{\tilde f} &&& Y \ar@{=}[r]&Y_0(z)\oplus Y_1(z)\oplus Y_2
	}
\]
Since $\tilde Y$ has fewer irrational ends than $Y$, the morphism $\tilde f:X\to \tilde Y$ has a kernel by induction. Since $h$ is a monomorphism by Claim 1, the kernel of $\tilde f$ is also the kernel of $f=(h\oplus 1\oplus1)\circ \tilde f$.

To prove Claim 2, it suffices to lift the component $X\to Y_0(z)$ of $f$ to $Y_0(z-\e)$. Since $f_0:X(z)\to Y_0(z)$ is zero, we can take its lifting $X(z)\to Y_0(z-\e)$ to be zero. The other components of $X$ have coordinates $(a,b)$ where $a\le z-\mu<z-\e$. Therefore the morphism $X_2\to Y_0(z)$ lifts to $Y_0(z-\e)$.

\ul{Claim 3}. Finally, suppose we are in Case 2. Then we claim that there is a morphism $\tilde f$ making the following diagram commute.
\[\xymatrixrowsep{18pt}\xymatrixcolsep{8pt}
\xymatrix{
\tilde X= X_0(z-\e)\oplus X_1(z)\oplus X_2\ar[rrr]^{\tilde{ f}}\ar[d]_{j\oplus1\oplus1}&&& \tilde Y=Y_0(z-\e)\oplus Y_1(z)\oplus Y_2\ar@{>->}[d]^{h\oplus 1\oplus1}\\
X=X_0(z)\oplus X_1(z)\oplus X_2\ar[rrr]^{ f'}\ar[d]^\cong_{g\oplus id_{X_2}}&&& Y=Y_0(z)\oplus Y_1(z)\oplus Y_2\ar[d]^{1}\\
X=X_0(z)\oplus X_1(z)\oplus X_2\ar[rrr]^{ f} &&& Y=Y_0(z)\oplus Y_1(z)\oplus Y_2
	}
\]
Here $g$ is the automorphism of $X(z)=X_0(z)\oplus X_1(z)$ given by Proposition \ref{two cases for maps in P_z}. Thus, the $X_1(z)\to Y_0(z)$ component of $f'$ is zero. The morphism $X_2\to Y_0(z)$ lifts to $Y_0(z-\e)$ as in the proof of Claim 2 and $\tilde f_0:X_0(z-\e)\to Y_0(z-\e)$ is the morphism which makes the following square Cartesian.
\[
\xymatrix{
X_0(z-\e)=M(z-\e,x_0)\ar[r]^{\tilde{f_0}}\ar[d]_j& Y_0(z-\e)=M(z-\e,y_0)\ar[d]^h\\
X_0(z-\e)=M(z,x_0)\ar[r]^{\tilde{f_0}}& Y_0(z-\e)=M(z,y_0)
}
\]
This square is Cartesian in $\cB$ and thus in $\cC$ (forms a triangle $X_0(z-\e)\to X_0(z)\oplus Y_0(z-\e)\to Y_0(z)\to X_0(z-\e)'$) since the horizontal maps are isomorphisms on the first coordinates and the vertical maps are isomorphisms in the vertical coordinates. This, in turn makes the top square in the previous diagram Cartesian in $\cC$. 

The morphism $\tilde f$ has a kernel, say $A$, by induction on the number of irrational ends in the target. Now we claim that the kernel of $\tilde f$ is equal to the kernel of $f'$. Since $j$ is a monomorphism in $\cC/\cT$, the composition $A\to \tilde X\to X$ is also a monomorphism. Thus it suffices to show that any map $k:Z\to X$ so that $f'\circ k:Z\to Y$ is zero in $\cC/\cT$ lifts to $A$. In the category $\cC$, the morphism $f'\circ k$ lifts to $T^Y$. But $h$ induces a split monomorphism $T^{\tilde Y}\to T^Y$ by construction and the missing part comes from $T^{X_0(z)}$. Therefore, $f'\circ k$ lifts to $X\oplus \tilde Y$. So, we can modify $k:Z\to X$ by a map which factors through $\cT$ and we have another map $Z\to \tilde Y$ which factors through $\cT$ so that their sum goes to zero in $Y$. By the fact that the upper square is Cartesian in $\cC$ we conclude that these lift to a map $\tilde k:Z\to \tilde X$ so that $\tilde f\circ\tilde k$ is trivial in $\cC/\cT$. So, $\tilde k$ lifts to $A$ as a morphism in $\cC/\cT$. This proves that $A$ is the kernel of $f'$ and thus also of $f$.

This completes the recursive construction of the kernel of $f$ in all cases.
\end{proof}

\begin{lem}
If $f:X\to Y$ is a morphism in $\cC/\cT$ then the following are equivalent.
\begin{enumerate}
\item $f$ is an epimorphism in $\cC/\cT$.
\item The induced mapping
\[
	\Hom_{\cC/\cT}(Y,\t S)\to \Hom_{\cC/\cT}(X,\t S)
\]
is a monomorphism for all $S\in\cT$.
\item The induced map
\[
	\Hom_{\cC/\cT}(\t^{-1}S,X)\to \Hom_{\cC/\cT}(\t^{-1}S,Y)
\]
is an epimorphism for all $S\in \cT$.
\end{enumerate}
\end{lem}

\begin{proof}
The equivalence between (1) and (2) is analogous to the corollary above. The equivalence between (2) and (3) follows from infinitesimal AR-duality.
\end{proof}

Putting together the above lemma and corollary we get the following.

\begin{prop}\label{mono and epi iff iso on support}
A morphism $f:X\to Y$ in $\cC/\cT$ is both a monomorphism and an epimorphism if and only if $f$ induces an isomorphism of supports, i.e., if
\[
	\Hom_{\cC/\cT}(\t^{-1}S,X)\to \Hom_{\cC/\cT}(\t^{-1}S,Y)
\]
is an isomorphism for all $S\in \cT$.
\end{prop}

For any $S\in\cT$ and $X\in \cC$ let $\Hom_0(S,X)$ denote the quotient of $\Hom_\cC(S,X)$ by all morphisms which factor through an object $T\in\cT$ which is not isomorphic to $S$. In particular, if $X\in add\cT$ then $\Hom_0(S,X)=K^n$ where $n$ is the number of summands of $T$ isomorphic to $S$. For general $X$ we have the following.

\begin{cor}
$\Hom_0(S,X)\cong \Hom_0(S,T_0^X)\cong K^n$ where $n$ is the number of times that $S$ occurs as a direct summand of $T_0^X$.
\end{cor}

\begin{lem}
For any $S\in\cT$ and $X$ in $\cC$ with no component isomorphic to $S$, there is a natural four term exact sequence
\[
	0\to \Hom_0(S,T_1^X)\to \Hom_{\cC/\cT}(\t^{-1}S,X)\xrarrow \f \Hom_{\cC/\cT}(\rt S,X)\to \Hom_0(S,X)\to 0
\]
\end{lem}

\begin{proof} The sequence is natural and additivity in $X$. So, it suffices to prove exactness when $X$ is indecomposable. $\Hom_{\cC/\cT}(\t^{-1}S,X)$ is nonzero if and only if $S$ lies on the walk associated to $X$ but is not one of the endpoints. This is because, for sufficiently small $\e>0$, $S_\e$ is the point right above and to the right of $S$. So, $S_\e$ lies in the rectangle $R_X$ except when $S$ is an endpoint of the walk associated to $X$ and the straight line from $S_\e$ to $X$ does not pass through the walk.

Since $\rt S$ is the direct sum of two objects, one just above $S$ and one just to the right of $S$, $\Hom_{\cC/\cT}(\rt S,X)$ will be nonzero for any object in the walk except for the points $A_i$ which are the summands of $T_1^X$ and it will have rank 2 for the points $B_i$ which are not endpoints. Therefore, the kernel of the map $\f:\Hom_{\cC/\cT}(\t^{-1}S,X)\to \Hom_{\cC/\cT}(\rt S,X)$ measures how many times $S$ occurs in $T_1^X$ and this is exactly $\Hom_0(S,T_1^X)$. And the cokernel of $\f$ is $\Hom_0(S,T_0^X)$ which is isomorphic to $\Hom_0(S,X)$ as long as $S$ is not a component of $X$.
\end{proof}

\begin{thm}
The categories $\cX/\cT$ and $\cC/\cT$ are abelian.
\end{thm}

\begin{proof}
It suffices to show that every morphism which is both a monomorphism and an epimorphism is an isomorphism. So, suppose that $f:X\to Y$ is such a morphism.  By Proposition \ref{mono and epi iff iso on support}, $f$ induces an isomorphism on supports. So, $X,Y$ have the same number of irrational ends. Suppose for a moment that $X,Y$ are both in $\cX/\cT$. Consider the 4-term sequences given by the lemma above:
\[
\xymatrix{
	0\ar[r]& \Hom_0(S,T_1^X)\ar[r]\ar[d]_\a& \Hom_{\cC/\cT}(\t^{-1}S,X)\ar[r]\ar[d]_\b &\Hom_{\cC/\cT}(\rt S,X)\ar[r]\ar[d]_\g& \Hom_0(S,X)\ar[r]\ar[d]_\delta &0\\
	0\ar[r]& \Hom_0(S,T_1^Y)\ar[r]& \Hom_{\cC/\cT}(\t^{-1}S,Y)\ar[r]&\Hom_{\cC/\cT}(\rt S,Y)\ar[r] & \Hom_0(S,Y)\ar[r] &0\\
	}
\]
Since $f$ induces an isomorphism on supports, $\b$ is an isomorphism. Since $\rt S$ represents a sequence of objects which do not lie in $\cT$, we have that $\g$ is a monomorphism
Chasing the diagram we conclude that $\a$ is an isomorphism, so that $T_1^X\cong T_1^Y$ and $\delta$ is a monomomorphism, so that $T_0^X$ is a direct summand of $T_0^Y$. Call the other summand $T_2$. Then the morphism $f$ fits into the following map of triangles:
\[
\xymatrix{
	T_1^X\ar[r]\ar[d]^\cong&T_0^X\ar[r]\ar[d]^j& X\ar[r]\ar[d]^f& T_1^X\ar[d]^\cong\\
	T_1^Y\ar[r] & T_0^Y\ar[r] & Y\ar[r] & T_1^Y \\	
}
\]
Since $T_1^X\cong T_1^Y$, the middle square is homotopy cartesian. So, $f$ completes to a triangle
\[
	X\to Y \to T_2\to X
\]
But this splits, with splitting given by $T_2\to T_0^Y\to Y$. So, $Y\cong X\otimes T_2$ where $T_2\in add\,\cT$. Therefore $f:X\to Y$ is an isomorphism in $\cX/\cT$ as claimed.

Now suppose that $X,Y$ are not in $\cX$. Then we proceed by induction on the number of irrational ends of $Y$. Going back to the proof of Proposition \ref{morphisms in C/T have kernels}, the first possibility  ($f_0=0$) cannot occur since that would make the cokernel of $f$ nontrivial. So, we are in the second case where we have a Cartesian square in $\cC$. By induction, the morphism $\tilde f:\tilde X\to \tilde Y$ in the proof of Proposition \ref{morphisms in C/T have kernels} is an isomorphism in $\cC/\cT$. But this implies that $X_0(z-\e)\cong Y_0(z-\e)$ since these are the only summands with an end at $z-\e$. Therefore $x_0=y_0$ and $X_0(z)\cong Y_0(z)$ and $f$ is also an isomorphism in $\cC/\cT$.
\end{proof}

\section{The infinite Jacobian algebra}

Let $\LL$ be the infinite algebra without unit defined to be the subring of $\End_\cX(T)$ consisting of all endomorphism which are zero on all but finitely many components of $T$ where $T$ is the direct sum of indecomposable objects of $\cT$, choosing one object from each isomorphism class. For example, we can take all objects to have positive parity in some fixed fundamental domain. This is the Jacobian algebra of an infinite quiver with potential given by taking the unique infinite simply connected planar trivalent tree and replacing each vertex with a triangle and each edge with a shared vertex of two triangles. Orient all the triangles clockwise and take the potential to be the sum of all of these 3-cycles. Then the Jacobian algebra $\LL$ is given by the infinite quiver $Q$ with the relation that the composition of any two arrows in the same triangle is zero.

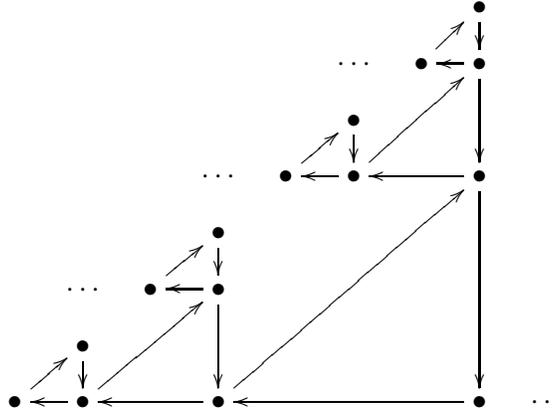
\[\xymatrixrowsep{10pt}\xymatrixcolsep{10pt}
\xymatrix{
&&&&&&&\bullet\ar[d]\\
&&&&& \cdots &\bullet\ar[ur]&\bullet\ar[l]\ar[dd]\\
&&&&& \bullet\ar[d] \\
&&& \cdots&\bullet\ar[ur] &\bullet\ar[l]\ar[uurr]&&\bullet\ar[ll]\ar[dddd]\\ 
&&&\bullet\ar[d]\\
& \cdots &\bullet\ar[ur]&\bullet\ar[l]\ar[dd]\\
& \bullet\ar[d] \\
\bullet\ar[ur] &\bullet\ar[l]\ar[uurr]&&\bullet\ar[ll]\ar[uuuurrrr]&&&&\bullet\ar[llll] & \cdots\\ 
	}
\]

\subsection{Finitely generated modules over the Jacobian}


\begin{prop}
Every finitely generated modules over $\LL$ is a direct sum of finitely many string modules which may be infinite dimensional.
\end{prop}

\begin{proof}
Choose a subquiver $Q'$ of $Q$ which is a connected union of triangles containing the supports of the generators of a f.g. module $M$. Let $I'$ be the ideal in $KQ'$ generated by the relations in $Q$ which are supported on $Q'$. Then $KQ'/I'$ is a string algebra by \cite{BR87}. So, the restriction $M'=M|Q'$ of $M$ to $Q'$ is a direct sum of finitely many string modules. We will say that a vertex of $Q'$ is a \emph{boundary point} if it lies in only one triangle of $Q'$. Let $Q''$ be obtained from $Q'$ by adding one triangle $u\to v\to w\to u$ to each boundary point $v$ of $Q'$. Let $M''=M|Q''$.

Claim 1. For each triangle $u\to v\to w\to u$ in $Q''$ with only one vertex $v$ in $Q'$, we have $M_u=0$ and the linear map $M_v\to M_w$ is an epimorphism.

Proof: $M$ is generated by elements of $KQ'$ but there are no paths from $Q'$ to $u$. Also, any element in the cokernel of $M_v\to M_w$ cannot be in the image in $M$ of a projective module with top in $Q'$. This proves the claim.

Let $d$ be the sum of the ranks of the maps $M_v\to M_w$ from the boundary points of $Q'$ to boundary points of $Q''$. Choose the pair $Q',Q''$ so that $d$ is minimal. Note that $d$ is the sum of the dimensions of $M_w$ over all boundary points $w$ of $Q''$. Let $Q'''$ be the subquiver of $Q$ obtained by adding triangles to the boundary points of $Q''$.

Claim 2. Let $u\to v\to w\to u$ be a triangle in $Q'''$ where $v$ is a boundary point of $Q''$. Then $M_v\to M_w$ is an isomorphism.

Proof: By Claim 1, the map $M_v\to M_w$ is surjective and $M_u=0$. So, the sum of the dimensions of $M_w$ is $\le d$. My minimality of $d$ we must have equality and the map must be an isomorphism as claimed.

Now take any decomposition of $M''$ into string modules. This gives a decomposition of $M_v$ into one dimensional subspaces for every boundary point $v$ of $Q''$. By Claim 2 we have isomorphisms $M_v\cong M_w$ which we use to give a compatible decomposition of $M_w$ into one dimensional summands. By repeating this process for larger and larger subquivers we obtain a decomposition of the infinite module $M$ into a direct sum of string modules as claimed.
\end{proof}

As observed in Claim 2 in the proof above, the infinite tails of finitely generated string modules over $\LL$ must be eventually oriented outward (towards the infinite end).

Let $mod\-\LL$ be the category of finitely generated modules over $\LL$ and let $f\el mod\-\LL$ denote the full subcategory of modules of finite length. Then $f\el mod\-\LL\cong f\el mod\-\hat\LL$ where $\hat\LL=\lim \LL/r^n\LL$ is the \emph{completion} of $\LL$. For any $\LL$-module $M$ let $FM$ be given by
\[
	FM=\lim_\to \Hom_\LL(\LL/r^n\LL,M)
\]
Viewing $\LL/r^n\LL$ as an $\hat\LL\-\LL$-bimodule we see that this gives a functor $F:mod\-\LL\to mod\-\hat\LL$.

\begin{lem}
$F$ is a left exact functor which vanishes on projective modules. Furthermore $FM$ has finite length for all f.g. $M$ and $FW=W$ for all modules $W$ of finite length.
\end{lem}

\begin{proof}
Since $\Hom$ is left exact and direct limit is exact, it follows that $F$ is left exact. Since a projective module $P$ contains no nonzero submodule of finite length we have $\Hom(\LL/r^n\LL,P)=0$ for all $n$. So $FP=0$. Also, the last claim: $FW=W$ for modules of finite length is clear.

Now suppose that $M$ is finitely generated and not projective. Then $M$ has 0,1 or 2 infinite ends and each infinite end is an infinite directed path starting at a vertices $x_i$. Since $M$ is not projective, these vertices $x_i$ must be distinct. Let $R(x_i)$ be the infinite string modules supported on these infinite paths. Then we have a short exact sequence
\[
	0\to W\to M\to \oplus R(x_i)\to 0
\]
for some module $W$ of finite length.
Since $\Hom_\LL(\LL/r^n\LL,R(x_i))=0$ for all $n$, we get by left exactness of $F$ that $FM\cong FW$. But $FW=W$. So we are done.
\end{proof}

This lemma shows that $F$ induces a functor on the stable category
\[
	\ul F:\ul{mod\-\LL}\to f\el mod\-\hat\LL= f\el mod\- \LL
\]
which can be viewed as a retraction since it has a section $f\el mod\-\LL\into \ul{mod\-\LL}$. However, $\ul F$ is not an isomorphism since there is no stable morphism $M\to FM$ in general. However, we have the following observation.

\begin{prop}
Let $\cX$ be the full subcategory of the stable category $\ul{mod\-\LL}$ generated by all string modules with two infinite ends. Then $\ul F$ induces an isomorphism $\cX\cong f\el mod\-\LL$.
\end{prop}

\begin{proof}
There is an inverse functor $G:f\el mod\-\LL\to \cX$ which is given by sending each string module of finite length $W$ to the unique string module $M$ having two infinite ends so that $FM=W$. The key point is that
\[
	\ul\Hom(M_1,M_2)\cong \Hom(W_1,W_2)
\]
which is verified by examining all possible ways that the stings might intersect. The intersection is again a string module since $Q$ is simply connected (when we fill in the triangles).
\end{proof}

\begin{thm}\label{thm: Q/T = fl mod L}
The abelian category $\cX/\cT$ is isomorphic to $f\el mod\-\LL$.
\end{thm}

\begin{proof}
\ul{Claim 1}. There is a bijection between the indecomposable objects of the two categories. 

Take any string module $W$ of finite length $n$. Then $W$ has support $v_1,\cdots,v_n$. We have the corresponding string module $GW=M$ which has two infinite ends starting at two distinct vertices $a,b$. These are two sources in the support of $M$. The support of $W$ gives a sequence of vertices connecting $a$ and $b$ by a zig-zag path:
\[
	\infty\- \text{end}\ot a\to v_1-v_2 \cdots v_n\ot b\to \infty\-\text{end}
\]
These correspond to objects $T_a,T_{v_i},T_b$ of $\cT$ which map nontrivially to each other in the opposite direction:
\begin{equation}\label{zig-zag support of M}
	T_a\ot T_{v_1}\cdots T_{v_n}\to T_b
\end{equation}

Since $Q$ is simply connected, $\Hom(T_a,T_b)=0=\Hom(T_b,T_a)$. In terms of coordinates, $T_a=(a_1,a_2), T_b=(b_1,b_2)$ where by symmetry we may assume $a_1<b_1,a_2>b_2$.

\ul{Claim 2}. There are no other objects of $\cT$ in the rectangle $[a_1,b_1]\times [b_2,a_2]$ except for the objects $T_a,T_b$ and $T_{v_i}$.

The reason is that each arrow in the diagram (\ref{zig-zag support of M}) is irreducible and we know that all of the objects of $\cT$ in the rectangle $[a_1,b_1]\times [b_2,a_2]$ are connected by a sequence of irreducible maps. These represent two walks in the quiver $Q$ from $a$ to $b$ which do not go through any zero relation. But such a walk is uniquely determined by its end points. So, they are equal. This proves Claim 2.

The indecomposable object of $\cX/\cT$ associated to $W$ is $X=M(b_1,a_2)$. This is the universal object of $\cX$ to which all of the objects $T_a,T_b,T_{v_i}$ map nontrivially.

Conversely, take any indecomposable object $X=M(x,y)$ in $\cX$ which is not in $\cT$. Let $a_2=y$, $b_1=x$ and let $a_1$ be the largest real number $<x$ so that $M(a_1,a_2)\in\cT$ and let $b_1$ be the largest real number $<y$ so that $M(b_1,b_2)\in\cT$. Then the objects of $\cT$ in the rectangle $[a_1,b_1]\times [b_2,a_2]$ excluding the corners $T_a=M(a_1,a_2)$ and $T_b=M(b_1,b_2)$ form the support of a string module $W$ of finite length which will give back $X$ by the above construction. Thus this describes the inverse process and gives the desired bijection concluding the proof of Claim 1.

Define a \emph{standard string module} to be one for which, at each vertex in the support we have $K$ and for each arrow we have the identity map $K\to K$. A \emph{basic morphism} between string modules is one which is the identity map on $K$ on each vertex in the intersection of supports. Since $Q$ is simply connected, any nonzero morphism between standard string modules is a scalar multiple of a basic morphism and any composition of basic morphisms $X\to Y, Y\to Z$ is either a basic morphism $X\to Z$ if there is one or zero if there is no basic morphism $X\to Z$.

\ul{Claim 3}. Suppose that $W_1,W_2$ are standard string modules of finite length and $X_1,X_2$ are the corresponding objects of $\cX/\cT$. Then $\Hom_\LL(W_1,W_2)=\Hom_{\cX/\cT}(X_1,X_2)$. 

Since each side is at most one dimensional it suffices to show that a nonzero morphism on one side implies the existence of a nonzero morphism on the other. Suppose that $\Hom_\LL(W_1,W_2)=K$. Then the supports of $W_1,W_2$ must intersect in another string module $W_3$. Let $v,v'$ be the endpoints of the support of $W_3$. Take $v$. Either $v$ is an endpoint of the support of $W_1$ or there is another vertex $w$ in the support of $W_1$ and an arrow $v\to w$. This gives an irreducible map $T_w\to T_v$ in $\cT$. In $M_2$ there must be a vertex $a$ and an arrow $a\to v$ giving an irreducible map $T_v\to T_a$. This means that
\[
	T_w\to T_v\to T_a\to T_w'
\]
is a distinguished triangle where $T_w'=T_w[1]$ is $T_w$ with the opposite parity. By symmetry we may assume that this is a negative triangle so that $T_w,T_v$ have the same $y$-coordinate and $T_v,T_a$ have the same $x$-coordinate. Then all of the other vertices in the part of the support of $W_1$ in the complement of the support of $W_3$ and containing $w$ will lie to the north-west of $w$. Therefore, they lie to the south-west of $a$ and that makes the top of the rectangle for $X_1$ lie below the top of the rectangle for $X_2$ but above the point $v$. If $v$ is an endpoint of $W_1$ but not of $W_2$ then the point $a$ is the new source of $M_1$, assuming by symmetry that $T_v,T_a$ have the same $x$ coordinate, we conclude that the top of the rectangle for $X_1$ is at the $y$-coordinate of the point $a$ which is at or below the top of the rectangle for $X_2$.

A similar argument at the other endpoint of $W_3$ tells us that the right side of the rectangle for $X_1$ lies to the left or is equal to the right side of the rectangle for $X_2$. Therefore, $X_1$ lies inside the rectangle for $X_2$ and lies in the upper right side of the zig-zag created by the objects of $\cT$ going from lower right to upper left in the $X_2$-rectangle. Therefore, $\Hom_{\cX/\cT}(X_1,X_2)=K$ since $X_1$ is below and to the left of $X_2$ and there are no objects of $\cT$ between $X_1$ and $X_2$.

Finally, suppose that $\Hom_{\cX/\cT}(X_1,X_2)=K$. Then $X_1$ must be below and to the left of $X_2$. It must also be above and to the right of the zig-zag of objects of $\cT$ going from the lower right to upper left in the $X_2$ rectangle since these are all the points which map to $X_2$ in the category $\cX/\cT$. The construction of the string module $W_1$ consists of going to the left and down from $X_1$ to give another rectangle. The horizontal line through $X_1$ either meets an object of $\cT$ in the $X_2$ rectangle or it goes to an object $T_v$ outside and to the left of this rectangle. But in that case, at the object $T_w$ where the two zig-zags meet, we have a morphism into $T_w$ from the support of $X_1$ which is an arrow away from $w$ in the support of $W_1$ which means that $W_1$ maps to $W_2$. (We need to do the same analysis at the other end to confirm this.) Therefore, $\Hom_\LL(W_1,W_2)=K$ as claimed. This proves Claim 3. The theorem follows since composition of basic morphisms is given by the same rule in both categories. (The composition of basic morphisms is a basic morphism if one exists.)
\end{proof}

\subsection{Infinitely generated modules}

The next theorem identifies the abelian category $\cC/\cT$ as a category of infinitely generated modules over $\LL$. Let $Mod$-$\LL$ denote the category of all locally finite $\LL$-modules. These are the $\LL$-modules $M$ so that $M_v$ is finite dimensional for all vertices $v\in Q_0$.

Let $Rep_+\LL$ denote the additive full subcategory of $Mod$-$\LL$  generated by all string module $M$ having two infinite ends so that in each end, the arrows are either eventually all pointing outward (as when $M$ is projective) or have an infinite number of arrows going inward and outward. The disallowed ends, the ones where the arrows are eventually all pointing inward, will be called ``injective ends'' or ends of ``injective type.''

Let $Rep_0\LL$ denote the additive full subcategory of $Mod$-$\LL$ category of string modules of $\LL$ with either zero, one or two infinite ends so that each infinite end has an infinite number of arrows in each direction. In other words, the direction of the arrows keeps switching back and forth.

For such a string module $W$ in $Rep_0\LL$ let $GW=M$ be the module with two infinite ends where we do the previous construction on any finite end. Thus, at each finite end $v$ we attach the unique vertex $w$ so that the triangle containing $v$ and $w$ meets the support of $W$ only at $v$. Then we add the ray starting at $w$ and going the other way to get $M$. As before, this has the property that
\[
	FM=\lim_\to \Hom_\LL(\LL/r^n\LL,M)=W
\]

\begin{prop}\label{prop: description of modules M=GW}
A string module $M$ has the form $M=GW$ for some $W\in Rep_0\LL$ if and only if $M$ is not projective and if $M$ has two infinite ends which are not of injective type.\qed
\end{prop}

Let $v$ be a vertex in the support of $M=GW$ and suppose that $T_v=M(a,b)$ is the corresponding object of $\cT$. Suppose by symmetry that $b\ge a$. Also, suppose that $b-a\ge0$ is minimal among all the vertices in the support of $M$. Then $a\le b<a+\pi$ and
\[
	\th=a+\pi-b=\frac{\pi}{2^n}
\]
for some nonnegative integer $n$. Furthermore, $a$ and $b$ are both integer multiples of $\th$. The cluster $\cT$ contains two other points $T_{v0}=M(a-\th/2,b)$ and $T_{v1}=M(a,b+\th/2)$ and there are irreducible morphisms $T_{v0}\to T_v\to T_{v1}\to T_{v0}'$ forming a distinguished triangle. Thus the vertices $v0,v,v1$ form a triangle in $Q$ with arrows going the other way: $v0\ot v\ot v1\ot v0$.

Since $M$ is a string module containing $v$ in its support and having two infinite ends, the support of $M$ contains exactly one of $v0,v1$. Continuing in the same direction, we have a sequence of 0's and 1's:
\[
	vd_1d_2d_3\cdots
\]
where $d_2=0$ if the next vertex $vd_10$ points away from $vd_1$.

\begin{lem}
The object in $\cT$ corresponding to $vd_1d_2\cdots d_m$ has coordinates $a_m,b_m$ where
\[
	b_m=b+\sum_{i=1}^m d_i\th/2^i,\quad a_m=b_m-\pi+\th/2^m
\]
\end{lem}

Note that in the limit as $m\to\infty$ we get that $b_m$ monotonically increases to $b_\infty$ and $a_m$ monotonically decreases to $a_\infty=b_\infty-\pi$. There is one kind of sequence which is excluded: There cannot be an integer $N$ so that $d_i=1$ for all $i\ge N$ because this would correspond to the case when the arrows all point inward after that point. As a consequence, the sequence $d_1,d_2,\cdots$ is uniquely determined by $b_\infty$ (and $b$ and $\th$) and $b\le b_\infty<b+\th$.

If we go in the other direction, $a$ will become larger than $b$ by choice of the vertex $v$. There are three cases to consider.
\begin{enumerate}
\item $a=b$.
\item $a<b$ and $a$ increases in the second tail ($a_{-1}>a,b_{-1}=b$).
\item $a<b$ and $b$ decreases in the second tail ($a_{-1}=a,b_{-1}<b$).
\end{enumerate}

\ul{Case 1}. $a=b$. In the fundamental domain, this is the point $v=(0,0)$. We have $\th=a+\pi-b=\pi$. All points $(x,y)$ in $\cT$ with $x\le y$ are in the second quadrant up to isomorphism with $-\pi<x\le0\le y<\pi$. The first tail lies here with $0\le b_\infty<\pi$. The second tail must be in the fourth quadrant with $0\le x<\pi$ and $0\ge y>-\pi$. In Case 1 we have symmetry between $a$ and $b$ and in the vertices in the second tail the roles of $a,b$ are switched. We get a sequence of 0's and 1's:
\[
	v_{-m}=ve_1e_2e_3\cdots e_m=(a_{-m},b_{-m})
\]
where
\[
	a_{-m}=-\sum_{i=1}^me_i\pi/2^i,\quad b_{-m}=a_{-m}-\pi+\pi/2^m
\]
As before, the binary digits $e_i$ are not allowed to all become equal to one after any point. Thus the $e_i$ are uniquely determined by the limits $a_{-\infty}$ and $b_{-\infty}=a_{-\infty}-\pi$ and
\[
	-\pi\le b_{-\infty}<0\le b_\infty<\pi,\quad 0\le a_{-\infty}<\pi\le a_\infty<2\pi
\]

\ul{Case 2}. $a<b$ and, if we go one step in the other direction, $a$ becomes greater than $b$ and $b_{-1}=b$. Since $\th<\pi$, either $(a+\th,b)$ or $(a,b-\th)$ lies in $\cT$ and lies in the second quadrant. In Case 2, it must be the latter and the closest point on the line $y=b$ to the right of $(a,b)$ is the point $(b+\pi-\th,b)$. This is equivalent to the point $(a-\th,b-\th)$ in the fundamental domain of $S$ and this translation of the second tail is in the triangle $\th$ units below and to the left of the triangle which contains all possible locations for the first tail. Since $S^{-1}(x,y)=(y-\pi,x-\pi)$, the coordinates of the limit point in this second triangular region are $(b_{-\infty}-\pi,a_{-\infty}-\pi)$. 
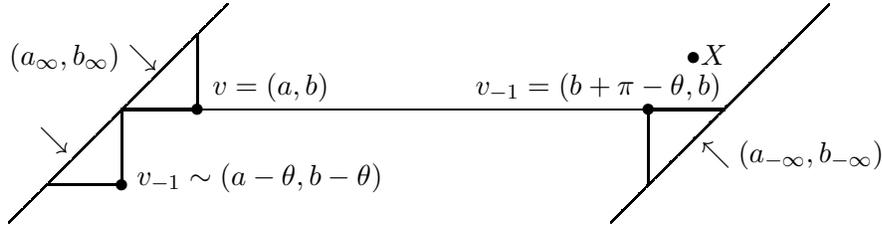
\begin{figure}[htbp]
\begin{center}
%
{
\setlength{\unitlength}{1cm}
{\mbox{
\begin{picture}(10,2.6)
\put(0,.3){
    \thinlines
      \qbezier(-.5,-.5)(1,1)(2.4,2.4)
      \qbezier(7.5,-.5)(9,1)(10.4,2.4)
      \thicklines
    \put(1,0){
          \qbezier(0,0)(-.5,0)(-1,0)
          \qbezier(0,0)(0,.5)(0,1)    
          \put(-.1,-.1){$\bullet$}
          \put(.2,0){$v_{-1}\sim(a-\th,b-\th)$}
          \put(-1.1,.5){$\searrow$}
          }
    \put(2,1){
          \qbezier(0,0)(-.5,0)(-1,0)
          \qbezier(0,0)(0,.5)(0,1)    
          \put(-.1,-.1){$\bullet$}
          \put(.2,.2){$v=(a,b)$}
          \put(-2.5,.6){$(a_{\infty},b_{\infty})\searrow$}
          }
        \put(8,1){
          \qbezier(0,0)(.5,0)(1,0)
          \qbezier(0,0)(0,-.5)(0,-1)    
          \put(-.1,-.1){$\bullet$}
          \put(-2.3,.2){$v_{-1}=(b+\pi-\th,b)$}
          \put(.7,-.7){$\nwarrow(a_{-\infty},b_{-\infty})$}
          \put(.5,.6){$\bullet X$}}
          \thinlines
          \qbezier(2,1)(5,1)(8,1)
          }
\end{picture}}
}}
\caption{Case 2: One tail is in the triangular region with corner $(a,b)$. The other tail is in the triangular region with corner $(b+\pi-\th,b)$ which is equivalent to $(a-\th,b-\th)$. $X$ is the corresponding object of $\cC/\cT$.}
\label{Case2}
\end{center}
\end{figure}
Both of these triangular regions are in the second quadrant. Thus, in Case 2 we have:
\[
	0\le b-\th\le b_{-\infty}<b\le b_\infty <b+\th\le\pi
\]
\[
	a-\th\le a_\infty<a<b+\pi-\th\le a_{-\infty}<b+\pi
\]

\ul{Case 3}. This is the same as Case 2 with the two tails reversed.

There is also Case 4 and Case 5 in the fourth quadrant which are similar to Cases 2 and 3 with the $x,y$-coordinates reversed.

\begin{prop}
The objects corresponding to the vertices in one tail of $GW=M$ converge to an element of $\RR/2\pi\ZZ$ of the form $a\pi/2^m$ if and only if the tail is eventually outward pointing.
\end{prop}

\begin{rem}
Near the line $y=x+\pi$ (the diagonal line on the left in Figure \ref{Case2}), these tails correspond to horizontal lines. Near the line $y=x-\pi$ (the diagonal line on the right in Figure \ref{Case2}), these outwardly pointing tails are vertical lines. Note that the vertical line below $v_{-1}=(b+\pi-\th,b)$ is equal to the horizontal line to the left of $v_{-1}=(a-\th,b-\th)$ in Figure \ref{Case2}.
\end{rem}

\begin{proof}
This follows from the fact that we have excluded the possibility of sequences ending in an infinite number of 1's since these would correspond to an infinite sequence of inwardly pointing arrows. Thus every binary sequence determines a unique real number which has the form $a\pi/2^m$ if and only if it has only finitely many 1's and the rest 0's which correspond to an infinite sequence of outwardly pointing arrows.
\end{proof}

\begin{thm}\label{thm: C/T=string modules without injective tails}
$Rep_0\LL$ is isomorphic to $\cC/\cT$.
\end{thm}

The argument is analogous to the finite case.

\begin{lem} There is a bijection between the string modules in $Rep_0\LL$ and the isomorphism classes of indecomposable objects of $\cC/\cT$.
\end{lem}

\begin{proof}
Take any object $W$ of $Rep_0\LL$. Form $M=GW$ as described above. The the corresponding object of $\cC/\cT$ is the point $X$ with coordinates $(a_{-\infty},b_\infty)$. The vertices in the support of $M$ form an infinite zig-zag starting at the right at the line $y=x-\pi$ and going left and up to the line $y=x+\pi$. The limiting horizontal line is $y=b_\infty$ at the top of the zig-zag and the limiting vertical line is $x=a_{-\infty}$ at the right of the zig-zag. In the case when $W$ has finite length, these form the top and right side of the rectangle $[a_1,b_1]\times [b_2,a_2]$ in the proof of Theorem \ref{thm: Q/T = fl mod L}.

Going case-by-case, the objects $M=GW$ of Case 1 correspond to the nonzero objects $X=M(x,y)$ of $\cC/\cT$ where $0\le x,y<\pi$. These are the points with the property that $\Hom(M(0,0),X)\neq0$. Then $y=b_\infty\in[0,\pi)$ and $x=a_{-\infty}$ is also in the half-open interval $[0,\pi)$. In Cases 2 and 3, $X=M(x,y)=M(a_{-\infty},b_\infty)$ lies in the second quadrant. In Cases 4 and 5, $X$ lies in the fourth quadrant.

The disallowance of injective tails gives a bijection between possible ends and the points on the unit circle. Pairs of distinct points $a,b$ on the unit circle correspond bijectively to the objects $M(a,b+\pi)$ of $\cC$. And the object lies in $\cT$ if and only if the corresponding string module is projective. By Proposition \ref{prop: description of modules M=GW} projective modules are the only modules $M$ without injective tails which cannot occur as $GW$ for any object $W$ of finite length. Thus we have the desired bijection.
\end{proof}

\begin{proof}[Proof of Theorem \ref {thm: C/T=string modules without injective tails}]
As in the proof of Theorem \ref {thm: Q/T = fl mod L}, it suffices to show that
\[
	\ul\Hom(M_1,M_2)=\Hom(W_1,W_2)=\Hom_{\cC/\cT}(X_1,X_2)
\]
where $X_1,X_2$ are the objects of $\cC/\cT$ corresponding to $M_1,M_2$. We can do this by reducing to the case when $W_i$ have finite length.

For each $W_i$ take a sequence of submodules $W_{in}$ of finite length so that $W_{in}$ containing any finite end of $W_i$ and also contains $n$ sinks in each infinite tail starting at some fixed vertex in the support of $W_i$. Let $X_{in}$ be the corresponding objects of $\cX/\cT$. Then, by Theorem \ref {thm: Q/T = fl mod L} we have
\[
	\Hom(W_{1n},W_{2m})=\Hom_{\cX/\cT}(X_{1n},X_{2m})
\]
Taking the limit as $m\to \infty$ we get
\[
	\Hom(W_{1n},W_2)=_{(1)}\lim_\to\Hom(W_{1n},W_{2m})=\lim_\to\Hom_{\cX/\cT}(X_{1n},X_{2m})=_{(2)}\Hom_{\cC/\cT}(X_{1n},X_{2})
\]
where (1) is clear and (2) follows from the fact that homorphisms vanish in a limit only when the limit line contains an object of $\cT$ or if it contains $TX_{1n}$ at the outer edge of the domain of the functor $\Hom_\cC(X_{1n},-)$. But this cannot happen since $X_{1n}$ has coordinates of the form $a\pi/2^k$ and limiting lines have coordinated with $\pi$ times a number with infinite binary expansion.

Now take the limit as $n\to \infty$, for which there is no problem since there are only finitely many objects of $\cT$ which can lie between $X_{1n}$ and $X_2$ at the point when $\Hom_{\cC/\cT}(X_{1n},X_2)\neq0$.
\end{proof}

\bibliographystyle{amsplain}
 

%
%
\end{document}